\documentclass[intlimits,12pt,reqno]{amsart}
\usepackage{latexsym,amsthm,amsmath,amssymb,mathrsfs,mathtools}
\usepackage[T1]{fontenc}
\usepackage[utf8]{inputenc}
\usepackage[mathscr]{eucal}
\usepackage{enumerate}
\usepackage{enumitem}

\let\oldmarginpar\marginpar
\renewcommand{\marginpar}[2][rectangle,draw,text width= 2cm,rounded corners]{
    \oldmarginpar{
    \scriptsize \tikz \node at (0,0) [#1]{#2};}
    }
\usepackage{tikz}
\usetikzlibrary{arrows,calc,patterns,cd}
\usepackage{wrapfig}

\usepackage{hyperref}
\usepackage{xcolor}
\hypersetup{
    colorlinks=true,
    linkcolor=blue,
    citecolor=blue,
    urlcolor=blue
}

plus 6pt minus 12pt
\def\mvint_#1{\mathchoice
          {\mathop{\vrule width 6pt height 3 pt depth -2.5pt
                  \kern -9pt \intop}\limits_{\kern -3pt #1}}%
          {\mathop{\vrule width 5pt height 3 pt depth -2.6pt
                  \kern -6pt \intop}\nolimits_{#1}}%
          {\mathop{\vrule width 5pt height 3 pt depth -2.6pt
                  \kern -6pt \intop}\nolimits_{#1}}%
          {\mathop{\vrule width 5pt height 3 pt depth -2.6pt
                  \kern -6pt \intop}\nolimits_{#1}}}

\newcommand{\Span}{\operatorname{span}}

\newcommand{\bbbr}{\mathbb R}

\newcommand{\bbbn}{\mathbb N}
\newcommand{\bbbz}{\mathbb Z}
\newcommand{\bbbq}{\mathbb Q}

\newcommand{\eps}{\varepsilon}

\def\diam{\operatorname{diam}}

\newtheorem{theorem}{Theorem}[section]
\newtheorem*{theorem*}{Theorem}
\newtheorem{lemma}[theorem]{Lemma}
\newtheorem{corollary}[theorem]{Corollary}
\newtheorem{proposition}[theorem]{Proposition}

\theoremstyle{definition}

\newtheorem{remark}[theorem]{Remark}
\newtheorem*{remark*}{Remark}

\newcommand{\supp}{{\rm supp}\,}

\newcommand{\wx}{\tilde{x}}
\newcommand{\wy}{\tilde{y}}

\newcommand*{\loc}{{\mathrm{loc}}}
\newcommand*{\PI}{{\mathrm{PI}}}

\usepackage{setspace}

\makeatletter
\renewcommand{\tocsection}[3]{%
  \indentlabel{\@ifnotempty{#2}{\bfseries\ignorespaces#1 #2\quad}}\bfseries#3}
\renewcommand{\tocsubsection}[3]{%
  \indentlabel{\@ifnotempty{#2}{\ignorespaces#1 #2\quad}}#3}

\newcommand\@dotsep{4.5}
\def\@tocline#1#2#3#4#5#6#7{\relax
  \ifnum #1>\c@tocdepth 
  \else
    \par \addpenalty\@secpenalty\addvspace{#2}%
    \begingroup \hyphenpenalty\@M
    \@ifempty{#4}{%
      \@tempdima\csname r@tocindent\number#1\endcsname\relax
    }{%
      \@tempdima#4\relax
    }%
    \parindent\z@ \leftskip#3\relax \advance\leftskip\@tempdima\relax
    \rightskip\@pnumwidth plus1em \parfillskip-\@pnumwidth
    #5\leavevmode\hskip-\@tempdima{#6}\nobreak
    \leaders\hbox{$\m@th\mkern \@dotsep mu\hbox{.}\mkern \@dotsep mu$}\hfill
    \nobreak
    \hbox to\@pnumwidth{\@tocpagenum{\ifnum#1=1\bfseries\fi#7}}\par
    \nobreak
    \endgroup
  \fi}
\AtBeginDocument{%
\expandafter\renewcommand\csname r@tocindent0\endcsname{0pt}
}
\def\l@subsection{\@tocline{2}{-5pt}{2.5pc}{5pc}{}}
\makeatother

\title[Reflexivity and separability of $N^{1,p}$]{A simple proof of reflexivity and separability of $N^{1,p}$ Sobolev  spaces}

\author[Alvarado]{Ryan Alvarado}
\address{Ryan Alvarado,\newline \indent Department of Mathematics and Statistics, Amherst College, 
\newline \indent 405 Seeley Mudd, Amherst,
Massachusetts 01002}
\email{rjalvarado\@@amherst.edu}

\author[Haj\l{}asz]{Piotr Haj\l{}asz}
\address{Piotr Haj\l{}asz,\newline \indent Department of Mathematics, University of Pittsburgh, \newline \indent 301 Thackeray Hall, Pittsburgh,
Pennsylvania 15260}
\email{hajlasz@pitt.edu}
\thanks{P.H. was supported by NSF grant  DMS-2055171.}

\author[Mal\'{y}]{Luk\'{a}\v{s} Mal\'{y}}
\address{Luk\'a\v{s} Mal\'y,\newline \indent Department of Science and Technology, Link\"oping University,
\newline \indent
 SE-601 74 Norrk\"oping, Sweden}
\email{lukas.maly\@@liu.se}


\begin{document}

\maketitle
\begin{abstract}
We present an elementary proof of a well-known theorem of Cheeger which states that if a metric-measure space $X$ supports a $p$-Poincar\'e inequality, then the $N^{1,p}(X)$ Sobolev space is reflexive and separable whenever $p\in (1,\infty)$. We also prove separability of the space when $p=1$.
Our proof is based on a straightforward construction of an equivalent norm on $N^{1,p}(X)$, $p\in [1,\infty)$, that is uniformly convex when $p\in (1,\infty)$. Finally, we explicitly construct a functional that is pointwise comparable to the minimal $p$-weak upper gradient, when $p\in (1,\infty)$.
\end{abstract}

\section{Introduction}
\label{intro}

Sobolev spaces on metric-measure spaces $M^{1,p}$ have been introduced in \cite{Haj2}, and soon after, many other definitions followed.
Independently, Cheeger \cite{Ch} and
Shanmugalingam \cite{Sha} introduced notions of Sobolev spaces on metric-measure spaces based on the upper gradient of Heinonen and Koskela \cite{HK}. 
Their spaces are denoted by $H_{1,p}$ and $N^{1,p}$, respectively.
While their definitions are different, it was observed by Shanmugalingam \cite[Theorem~4.10]{Sha}, that the spaces $H_{1,p}$ and $N^{1,p}$ are isometrically isomorphic when $p>1$.

Throughout the paper we assume that $(X,d,\mu)$ is a metric-measure space with a Borel regular doubling measure. In this setting,
we define $N^{1,p}(X)$, $p\in [1,\infty)$, as the space of functions $u\in L^p(X)$ that have an upper gradient in $L^p(X)$. $N^{1,p}(X)$ is a Banach space with respect to the norm
$$
\|u\|_{N^{1,p}(X)}:=\Big(\| u \|^p_{L^p(X)} + \inf_g \|g\|^p_{L^p(X)}\Big)^{1/p}.
$$
Here, the infimum is taken over all upper gradients $g$ of $u$. 
See Section~\ref{prelims} for additional details regarding our setting and the space $N^{1,p}(X)$.

If there are no rectifiable curves in $X$, then $g=0$ is an upper gradient of any function, and hence, $N^{1,p}(X)=L^p(X)$ isometrically. Therefore, in order to have a rich theory, we need a large family of rectifiable curves in $X$, which is guaranteed when the space supports a $p$-Poincar\'e inequality. Recall that the space $(X,d,\mu)$ supports a $p$-Poincar\'e inequality, $p\in [1,\infty)$, if the measure $\mu$ is doubling and there are constants 
$c_{\PI} > 0$ and $\lambda \geq 1$ such that
$$
\mvint_B |u-u_B|\,d\mu \le c_{\PI}\diam (B) \Biggl(\,\,\mvint_{\lambda B} g^p\,d\mu\Biggr)^{1/p}
$$
for all balls $B\subseteq X$, for all Borel functions $u\in L^1_\loc(X)$, and all upper gradients $g$ of $u$. Here, and in what follows, the barred integral stands for the integral average and $u_B:=\mvint_B u\, d\mu$ is the integral average of $u$ over the ball $B$. Also, $\diam (B)$ denotes the diameter of $B$, and $\lambda B$ stands for a ball concentric with $B$ and radius $\lambda$ times that of $B$.

Cheeger \cite{Ch} proved that if the space $(X,d,\mu)$ supports a $p$-Poincar\'e inequality for some $p\in (1,\infty)$, then the space $N^{1,p}(X)$ is reflexive. In fact, he proved in this setting that the space $N^{1,p}(X)$ can be equipped with an equivalent uniformly convex norm, from which reflexivity follows. 
His proof of reflexivity is, however, very difficult and based on the celebrated construction of a measurable differentiable structure. 
Later Keith \cite{Keith} proved the existence of a measurable differentiable structure and hence, reflexivity of $N^{1,p}(X)$, $p\in (1,\infty)$, under the so-called Lip-lip condition.
As demonstrated by Heinonen \cite[Section~12.5]{Hei07}, for general metric-measure spaces, $N^{1,p}(X)$, $p\in (1,\infty)$, need not be reflexive.

A different approach to reflexivity was provided by Ambrosio, Colombo, and Di Marino \cite{Ambrosio}. They proved reflexivity of $N^{1,p}(X)$, $p\in(1,\infty)$, under the assumptions that the metric space $X$ is metric-doubling, complete, and the measure $\mu$ is finite on balls. They did not, however, assume that the space supports a $p$-Poincar\'e inequality. 
In fact, they proved reflexivity of a Sobolev type space $W^{1,p}(X)$ whose definition is based on a notion of $p$-relaxed slope, and they proved that the space is equivalent to $N^{1,p}(X)$ under the given assumptions.
Their proof is actually quite difficult since it involves methods of mass-transportation, gradient flows, $\Gamma$-convergence, and Christ dyadic cubes, just to name a few. 
A simplification of this proof of reflexivity in the case when the space supports a $p$-Poincar\'e inequality was obtained by Durand-Cartagena and Shanmugalingam \cite{DurSha}; their proof follows arguments from \cite{Ambrosio} and, in particular, they use $\Gamma$-convergence and Christ dyadic cubes to construct an equivalent norm on $N^{1,p}(X)$ that is uniformly convex.

Recently, Eriksson-Bique and Soultanis \cite{ES}, proved reflexivity of $N^{1,p}(X)$, $p\in (1,\infty)$, under the assumption that the space has finite Hausdorff dimension. Their proof is quite difficult too.

The purpose of this paper is to provide a further simplification of the proof of reflexivity of $N^{1,p}(X)$ when $p\in (1,\infty)$ and the space supports a $p$-Poincar\'e inequality. In fact, we provide an explicit construction of an equivalent norm on $N^{1,p}(X)$, $p\in [1,\infty)$, which is uniformly convex when $p\in (1,\infty)$. Our arguments are based on ideas from \cite{Ambrosio} and also \cite{DurSha}, but our construction of the uniformly convex norm is direct and it does not require $\Gamma$-convergence nor Christ cubes.

A brief outline of our construction is below. All details can be found in Section~\ref{uc}.

For each $k\in\bbbz$, we select a covering of $X$ by balls $\{B_i^k\}_i$, of radii $2^{-k}$, such that the balls in the family $\big\{\frac{1}{5}B_i^k\big\}_i$ are pairwise disjoint. We say that balls $B_i^k$ and $B_j^k$ are neighbors if $\operatorname{dist}(B_i^k, B_j^k)<2^{-k}$, and we denote neighbors by $B_i^k\sim B_j^k$.
It follows from the doubling condition that the number of neighbors 
of a given ball $B_i^k$
is bounded by some constant $N\in\bbbn$ that is independent of $k$.

For each $x\in X$, there is a smallest index $i$ such that $x\in B_i^k$, and we write $B^k[x]:=B^k_i$. Then for $p\in [1,\infty)$ and $u\in L^1_{\rm loc}(X)$ we define
$$
|T_k u(x)|_p:=
2^k\Bigg(\sum_{j:B_j^k\sim B^k[x]} |u_{B^k[x]}-u_{B_j^k}|^p\Bigg)^{1/p},
$$
where the sum is taken over all neighbors of $B^k[x]=B_i^k$, and we set $|T_k u(x)|:=|T_k u(x)|_1$.
Finally, we equip $N^{1,p}(X)$, $p\in [1,\infty)$ with a new norm,
$$
\Vert u\Vert_{1,p}^* :=\Big(\Vert u\Vert_{L^p(X)}^p+\limsup_{k\to\infty} \big\Vert|T_{k} u|_p\big\Vert_{L^p(X)}^p\Big)^{1/p}.
$$
The main result of the paper reads as follows.
\begin{theorem}
\label{main}
Suppose that the space $(X,d,\mu)$ supports a $p$-Poincar\'e inequality for some $p\in [1,\infty)$. Then $\Vert\cdot\Vert_{1,p}^*$ is an equivalent norm on $N^{1,p}(X)$. 
Moreover, if $p\in(1,\infty)$, then the space $N^{1,p}(X)$ with the equivalent norm $\Vert\cdot\Vert_{1,p}^*$ is uniformly convex and hence, the space $N^{1,p}(X)$ is reflexive.
\end{theorem}
\noindent The notion of uniform convexity is recalled in Section~\ref{Clar}.

The construction of the norm $\Vert\cdot\Vert_{1,p}^*$ is different from, but related to, the constructions given in \cite{Ambrosio,DurSha}. Recall that their constructions were less direct, as they required $\Gamma$-convergence and Christ cubes. The equivalence of the norms when $p=1$ is, however, new.
As a corollary, we also prove
\begin{theorem}
\label{apes}
Suppose that the space $(X,d,\mu)$ supports a $p$-Poincar\'e inequality for some $p\in [1,\infty)$. Then the space $N^{1,p}(X)$ is separable.
\end{theorem}
It is well known that separability can be deduced from reflexivity when $p\in(1,\infty)$, see \cite{Ch}, but separability in the case $p=1$ seems to be new.

It follows from the proof of Theorem~\ref{main} (more specifically, Proposition~\ref{gradcomparable}) that if $p\in [1,\infty)$, then there is $C\geq 1$ such that
$$
C^{-1}\Vert g_u\Vert_{L^p(X)}\leq \limsup_{k\to\infty} \big\Vert|T_{k} u|_p\big\Vert_{L^p(X)}\leq C\Vert g_u\Vert_{L^p(X)},
$$
where $g_u$ is the minimal $p$-weak upper gradient of $u$. The next result shows not only a comparison of norms, but a pointwise comparison under the additional assumptions that $X$ is complete and $p>1$.

\begin{theorem}
\label{main2}
Suppose that the space $(X,d,\mu)$ is complete and supports a $p$-Poincar\'e inequality for some $p\in(1,\infty)$. Then there exists a constant  $C\geq 1$ such that for every $u\in N^{1,p}(X)$,
$$
C^{-1} g_u(x)\leq \limsup_{k\to\infty} |T_{k} u(x)|\leq C g_u(x)\quad
\text{for $\mu$-a.e. $x\in X$,}
$$
where $g_u\in L^p(X)$ denotes the minimal $p$-weak upper gradient of $u$.
\end{theorem}
\begin{remark}
Note that we could replace $|T_ku(x)|$ in Theorem~\ref{main2} by $|T_ku(x)|_p$, because the number of neighbors is bounded by $N$ and all norms in $\bbbr^N$ are equivalent. However, in Theorem~\ref{main} we have to work with $|T_ku|_p$ in order to guarantee uniform convexity of the norm.
\end{remark}

The paper is structured as follows. In Section~\ref{prelims} we fix notation used in the paper, recall basic definitions, and state known results that will be used in the subsequent sections. In Section~\ref{uc} we carefully explain the statement of the main result, Theorem ~\ref{main}, and we prove it. In Section~\ref{sepa} we prove Theorem~\ref{apes} and finally, in Section~\ref{ptwise} we prove Theorem~\ref{main2}.

\subsection*{Acknowledgements}
We would like to thank Nicola Gigli from whom we learned Proposition~\ref{reflextosep}, and Giorgio Metafune for helpful comments.

\section{Preliminaries}
\label{prelims}

\subsection{Notational conventions}
Let $\bbbz$ denote all integers and $\bbbn$ all (strictly) positive integers. By $C$ we denote a generic constant whose actual value may change from line to line.
For nonnegative quantities, $L,R\geq 0$, the notation $L \lesssim R$ will be used to express that there exists a constant $C>0$, perhaps dependent on other constants within the context, such that $ L \le CR$. If $L \lesssim R$ and simultaneously $R \lesssim L$, then we will simply write $L \approx R$ and say that the quantities $L$ and $R$ are \emph{equivalent} (or \emph{comparable}).

The characteristic function of a set $E$ will be denoted by $\chi_E$.

We assume that all function spaces are linear spaces over the field of real numbers.

We use a convention that the names ``Theorem'' and ``Proposition'' are reserved for new results, while well-known results and results of technical character are called ``Lemma'' or ``Corollary''. 

\subsection{Metric-measure spaces}
A metric-measure space is a triplet $(X,d,\mu)$ where $(X,d)$ is a metric space and  $\mu$ is a Borel measure such that $0<\mu(B)<\infty$ for every ball $B\subseteq X$. We will assume that $\mu$ is \textit{Borel regular}, in the sense that every $\mu$-measurable set is contained in a Borel set of equal measure. We will also assume that $\mu$ is \emph{doubling}, i.e., there is a constant $C_d\geq 1$ such that $\mu(2B) \le C_d \mu(B)$ for every ball $B \subseteq X$. 

We will need the following version of the Lebesgue differentiation theorem.
\begin{lemma}
\label{T5}
Assume that $\mu$ is a Borel regular doubling measure on $X$ and $u\in L^1_{\rm loc}(X)$. Then for $\mu$-a.e. $x\in X$ the following is true. If $\{B_i\}_i$ is a sequence of balls such that $x\in B_i$ for all $i$ and $\diam(B_i)\to 0$ as $i\to\infty$, then
\begin{equation}
\label{eq3}
\lim_{i\to\infty}\, \mvint_{B_i} u\, d\mu=u(x).
\end{equation}
\end{lemma}
Equality \eqref{eq3} is satisfied whenever $x$ is a Lebesgue point of $u$.
This result is well known if $\mu$ is the Lebesgue measure in $\bbbr^n$, but the standard proofs easily generalize to the case of metric-measure spaces equipped with a Borel regular doubling measure.

\subsection{Integrating along curves in metric spaces and modulus of the path family}

By a \emph{curve} in $X$, we mean a continuous mapping $\gamma\colon[a,b]\to X$. Given a curve $\gamma$, the \textit{image of} $\gamma$ is denoted by $|\gamma|:=\gamma([a,b])$ and $\ell(\gamma)$ stands for the \textit{length of} $\gamma$. We will say that $\gamma$ is \textit{rectifiable} if $\ell(\gamma)<\infty$ and the family of all non-constant rectifiable curves in $X$ will be denoted by $\Gamma(X)$. Every $\gamma\in\Gamma(X)$ admits a unique (orientation preserving) \textit{arc-length parameterization} $\widetilde{\gamma}\colon[0,\ell(\gamma)]\to X$, and the arc-length parameterization is 1-Lipschitz; see, e.g., \cite[Theorem~3.2]{Haj}. Given a curve $\gamma\in\Gamma(X)$ and a Borel measurable function $\varrho\colon|\gamma|\to[0,\infty]$, we define
$$
\int_\gamma\varrho\,ds:=\int_0^{\ell(\gamma)}\varrho(\widetilde{\gamma}(t))\,dt.
$$
We can naturally define the integral over a curve for a general function by considering the positive and negative parts of the function.

Let $\Gamma\subseteq\Gamma(X)$ and consider the collection $F(\Gamma)$ of all Borel  functions $\varrho\colon X\to[0,\infty]$ satisfying
$$
\int_\gamma\varrho\,ds\geq1\quad\mbox{for all $\gamma\in\Gamma$.}
$$
Then, for each $p\in[1,\infty)$, the $p$-\textit{modulus of the family} $\Gamma$ is defined as
$$
{\rm Mod}_p(\Gamma):=\inf_{\varrho\in F(\Gamma)}\int_X\varrho^p\,d\mu.
$$
Note that ${\rm Mod}_p$ is an outer-measure on $\Gamma(X)$ and, in particular, it is countably subadditive; see, e.g., \cite[Theorem~5.2]{Haj}. A family of curves $\Gamma\subseteq\Gamma(X)$ is called \emph{$p$-exceptional} if ${\rm Mod}_p(\Gamma)=0$ and a statement is said to hold for \emph{${\rm Mod}_p$-a.e.} curve $\gamma\in\Gamma(X)$ if the family of curves in $\Gamma(X)$ for which this statement does not hold is $p$-exceptional.

For the next result, see \cite[Proposition~2.45]{BjoBjo}.
It follows from H\"older's inequality and \cite[Proposition~1.37(c)]{BjoBjo}.
\begin{lemma}
\label{except}
If a family of curves is $p$-exceptional for some $p\in(1,\infty)$, then it is
$q$-exceptional for every $q\in[1,p]$.
\end{lemma}

We will also need the following important result; see, e.g., \cite[Theorem~5.7]{Haj} and \cite[Lemma~2.1]{BjoBjo}.
\begin{lemma}[Fuglede's lemma]
\label{fuglede}
Let $p \in [1, \infty)$ and assume that $\{g_k\}_{k=1}^\infty$ is a sequence of Borel functions that converges in $L^p(X)$ to a  Borel function $g\in L^p(X)$. Then, there is a subsequence  $\{g_{k_i}\}_{i=1}^\infty$, such that for ${\rm Mod}_p$-a.e.\@ curve $\gamma \in \Gamma(X)$, one has
\[
\int_\gamma g_{k_i}\,ds \to \int_\gamma g\,ds\quad \text{and} \quad \int_\gamma |g_{k_i} - g|\,ds \to 0\quad\mbox{as $i\to\infty$,}
\]
where all of the integrals are well defined and  finite.
\end{lemma}

\subsection{Sobolev spaces in metric-measure spaces}
\label{subsect:sobolev}
A Borel measurable function $g\colon X \to [0, \infty]$ is called an \emph{upper gradient} of a Borel measurable function $u\colon X \to [-\infty,\infty]$ if
\begin{equation}
\label{eq:ug_def}
|u(\gamma(a)) - u(\gamma(b))| \le \int_\gamma g\,ds,
\end{equation}
for every rectifiable curve $\gamma\colon [a,b]\to X$, with the convention that $|(\pm\infty)-(\pm\infty)|=\infty$. The function $g$ shall be referred to as a \emph{$p$-weak upper gradient} of $u$, $p\in[1,\infty)$, if  \eqref{eq:ug_def} holds true for ${\rm Mod}_p$-a.e.\@ curve $\gamma\in\Gamma(X)$.

The next result shows that $p$-weak upper gradients can be approximated by upper gradients in the $L^p$ norm; see e.g. \cite[Lemma~6.3]{Haj}
\begin{lemma}
\label{T6}
If $g$ is a $p$-weak upper gradient of $u$ which is finite $\mu$-a.e., then for every $\eps\in (0,\infty)$ there is an upper gradient $g_\eps$ of $u$ such that
$$
g_\eps\geq g\, \text{ pointwise everywhere in $X$}
\quad
\text{and}
\quad
\Vert g_\eps-g\Vert_{L^p(X)}<\eps.
$$
\end{lemma}

For $p\in[1,\infty)$ we define $\widetilde{N}^{1,p}(X)$, to be the space of all Borel measurable functions $u: X \to [-\infty,\infty]$ for which
\begin{equation}
  \label{eq:def-N1p-norm}
\|u\|_{N^{1,p}(X)}:=\Big(\| u \|^p_{L^p(X)} + \inf_g \|g\|^p_{L^p(X)}\Big)^{1/p} < \infty,
\end{equation}
where the infimum is taken over all upper gradients $g$ of $u$. 
Equivalently,
we can take the infimum over all $p$-weak upper gradients in \eqref{eq:def-N1p-norm} since every $p$-weak upper gradient can be approximated in $L^p$ by upper gradients (Lemma~\ref{T6}).

The functional $\|\cdot\|_{N^{1,p}(X)}$ is a seminorm on $\widetilde{N}^{1,p}$ and a norm on ${N}^{1,p}(X) := \widetilde{N}^{1,p}(X)/\mathord\sim$, where the equivalence relation $u\sim v$ is given by $\|u-v\|_{N^{1,p}(X)} = 0$. Furthermore, the  space ${N}^{1,p}(X)$ is complete and thus a Banach space, see \cite[Theorem~3.7]{Sha}.

For the next result see, e.g., \cite[Corollary~7.7]{Haj}.
\begin{lemma}
\label{T3}
If $u,v\in\widetilde{N}^{1,p}(X)$ and $u=v$ pointwise $\mu$-a.e. in $X$, then $u\sim v$ i.e., the two functions define the same element in $N^{1,p}(X)$.
\end{lemma}

For $p\in[1,\infty)$, every $u\in N^{1,p}(X)$ has a \textit{minimal $p$-weak upper gradient} $g_u\in L^p(X)$ in the sense that if $g\in L^p(X)$ is another  $p$-weak upper gradient of $u$, then $g\geq g_u$ pointwise $\mu$-a.e. in $X$, see, e.g., \cite[Theorem~7.16]{Haj}.  Hence, the infimum in \eqref{eq:def-N1p-norm} is attained with $g_u$, which is given uniquely up to pointwise a.e. equality.

Recall that the \emph{pointwise lower Lipschitz-constant} of a function $\eta\colon X\to\mathbb{R}$ is given by
\begin{equation}
\label{lillip}
{\rm lip}\,\eta(x):=\liminf_{r\to 0^+}\sup_{y\in B(x,r)}\frac{|\eta(x)-\eta(y)|}{r},\quad x\in X.    
\end{equation}
For the next lemma, see, e.g., \cite[Lemma~6.7]{Haj} or \cite[Lemma~6.2.6]{HKST}.
\begin{lemma}
\label{T1}
${\rm lip}\,\eta$ is an upper gradient of any Lipschitz continuous function $\eta$ on a metric space.
\end{lemma}

\begin{lemma}
\label{leibniz}
Fix $p\in[1,\infty)$ and suppose that $u\in N^{1,p}(X)$ and $\eta\colon X\to\mathbb{R}$ is a bounded Lipschitz function. Then $\eta u\in N^{1,p}(X)$ and the function $h:=|\eta|g_u+|u|\,{\rm lip}\,\eta$ is a $p$-weak upper gradient for $\eta u$, where $g_u\in L^p(X)$ is the minimal $p$-weak upper gradient of $u$.
\end{lemma} 

\begin{proof}
The Leibniz rule for $p$-weak upper gradients, \cite[Lemma~6.3.28]{HKST}, the fact that functions in $N^{1,p}(X)$ are absolutely continuous on $\operatorname{Mod}_p$-a.e. curve, \cite[Lemma~7.6]{Haj}, and Lemma~\ref{T1}
imply that the function $h:=|\eta| g_u+|u|\operatorname{lip} \eta$
is a $p$-weak upper gradient for $\eta u$. Since $\eta u\in L^p(X)$ and $h\in L^p(X)$, it follows that $\eta u\in N^{1,p}(X)$.
\end{proof}

We say that $(X,d,\mu)$ supports a \emph{$p$-Poincar\'e inequality}, $p\in[1,\infty)$, if there exist constants $c_{\PI} > 0$ and $\lambda \geq 1$ such that
\begin{equation}
\label{poincare}
\mvint_B |u-u_B|\,d\mu \le c_{\PI}\diam (B) \Biggl(\,\,\mvint_{\lambda B} g^p\,d\mu\Biggr)^{1/p}
\end{equation}
for all balls $B\subseteq X$,  all Borel functions $u\in L^1_\loc(X)$, and all upper gradients $g$ of $u$. Recall that we always assume that $\mu$ is a Borel regular doubling measure in this setting.
In this situation we say that the space supports a $p$-Poincar\'e inequality with constants $c_{\PI}$ and $\lambda$.

The following lemma is an immediate consequence of 
Lemma~\ref{T6}.
\begin{lemma}
\label{pwkpoin}
Suppose that $X$ supports a $p$-Poincar\'e inequality for some $p\in[1,\infty)$. 
If $u\in L^1_{\rm loc}(X)$ is Borel, then
\begin{equation}
\label{poin-p-weak}
\mvint_B |u-u_B|\,d\mu \le c_{\PI}\diam (B) \Biggl(\,\,\mvint_{\lambda B} g^p\,d\mu\Biggr)^{1/p},
\end{equation}
for all balls $B\subseteq X$ and all $p$-weak upper gradients $g$ of $u$ that are finite $\mu$-a.e.
\end{lemma}

\subsection{Uniformly convex spaces} 
\label{Clar}
We begin with a definition due to Clarkson \cite{Clark}.

We say that a normed space $(Z,\|\cdot\|)$ is \emph{uniformly convex} if for every $\eps\in(0,\infty)$, there exists $\delta\in(0,\infty)$ with the property that $\|x + y\| \le 2(1-\delta)$ whenever $x,y\in Z$ satisfy $\|x\|=\|y\| =1$, and $\|x - y\| > \eps$. 

From a geometric point of view, uniform convexity implies that the boundary of the unit ball does not contain any segments and that the unit ball is, in a sense, uniformly ``round''. 

The next result is well known, but it is not easy to find a proof in the literature.
\begin{lemma}
\label{unifconvex-clsd}
A normed space $(Z,\|\cdot\|)$ is uniformly convex if and only if for every $\eps\in(0,\infty)$, there exists $\delta\in(0,\infty)$ with the property that $\|x + y\| \le 2(1-\delta)$ whenever $x,y\in Z$ satisfy $\|x\|\leq1$, $\|y\|\leq 1$, and $\|x - y\| > \eps$.  
\end{lemma}

\begin{proof}
One direction is clear. To see the other, fix $\varepsilon\in(0,\infty)$ and suppose that $x,y\in Z$ satisfy $\|x\|,\|y\|\leq1$ and $\|x-y\|>\varepsilon$. Since $Z$ is uniformly convex, there is $\tilde{\delta}\in(0,\infty)$ associated to the choice of $\varepsilon/3$. Let $\delta:=\min\{\varepsilon/6,\tilde{\delta}/3\}$. If either $\|x\|\leq1-2\delta$ or $\|y\|\leq1-2\delta$, then  $\|x+y\|\leq2(1-\delta)$. If $\|x\|,\|y\|>1-2\delta$, then  $\tilde{x}:={x}/{\|x\|}$ and $\tilde{y}:={y}/{\|y\|}$ satisfy $\Vert x-\tilde{x}\Vert,\Vert y-\tilde{y}\Vert<2\delta$, and hence, $\|\tilde{x}-\tilde{y}\|>\eps-4\delta\geq \varepsilon/3$. 
Since $\Vert\tilde{x}\Vert=\Vert\tilde{y}\Vert=1$, uniform convexity yields
$\|\tilde{x}+\tilde{y}\|\leq2(1-\tilde{\delta})$ and hence, $\|x+y\|\leq\|x-\wx\|+\|\wx+\wy\|+\|
y-\wy\|\leq2(1-\delta)$. 
\end{proof}

A clever proof of the next result that avoids the use of Clarkson's inequalities can be found in
\cite[Proposition~2.4.19]{HKST}.
\begin{lemma}[Clarkson]
\label{Lp}
$L^p(X)$ is uniformly convex for $p\in(1,\infty)$.
\end{lemma}
For a proof of the following theorem, see, e.g., \cite[Theorem~2.4.9]{HKST}.
\begin{lemma}[Milman--Pettis' theorem]
\label{milmanpettis}
Every uniformly convex Banach space is reflexive.
\end{lemma}

By $\ell^p_M$ we will denote $\bbbr^M$ with the norm $|x|_p:=\big(\sum_{j=1}^M|x_j|^p\big)^{1/p}$, where $x=(x_1,\ldots,x_M)$, and so $L^p(X,\ell^p_M)$ is a Banach space equipped with the norm
\begin{equation}
\label{reoi}
\Phi(f):=\Bigg(\sum_{j=1}^M\|f_j\|_{L^p(X)}^p\Bigg)^{1/p}, 
\quad
\text{where}
\quad
f=(f_1,\ldots,f_M).
\end{equation}
\begin{corollary}
\label{clarkson}
If $p\in(1,\infty)$ and $M\in\bbbn$, then $L^p(X,\ell^p_M)$ is uniformly convex.
\end{corollary}
\begin{proof}
Since $L^p(X,\ell^p_M)$ is isometric to $L^p(X_M)$, where 
$$
X_M=X\sqcup X\sqcup\ldots\sqcup X=X\times\{1,2,\ldots,M\}
$$ 
is the disjoint union of $M$ copies of the measure space $(X,\mu)$, the result follows from Lemma~\ref{Lp}.
\end{proof}

\subsection{Dunford-Pettis theorem}

Recall that a family of $\mu$-measurable functions $\mathcal{F}$ is said to be \emph{equi-integrable} if for every $\eps\in(0,\infty)$ there exists $\delta\in(0,\infty)$ such that for every $\mu$-measurable set $S \subseteq X$ with $\mu(S) < \delta$ we have
\begin{equation}
\label{it:equiint2}
\sup_{f\in\mathcal{F}}\int_S |f|\,d\mu < \eps.
\end{equation}

The proof for the following version of the Dunford--Pettis theorem can be found in, e.g., \cite[Theorem~2.54]{FL07}.

\begin{lemma}[Dunford--Pettis' theorem]
\label{thm:DunfordPettis}
Let $\mathcal{F}\subseteq L^1(X)$. 
Then every sequence in $\mathcal{F}$ has a subsequence that is weakly convergent in $L^1(X)$ if and only if
the following two conditions are satisfied:
\begin{enumerate}
\item[{\rm(a)}] $\mathcal{F}$ is bounded in $L^1(X)$ and equi-integrable;
\vskip.08in
\item[{\rm (b)}] for every $\eps\in(0,\infty)$ one can find a $\mu$-measurable set $E\subseteq X$ such that $\mu(E)<\infty$ and
\begin{equation}
\label{it:equiint1}
\sup_{f\in\mathcal{F}}\,\int_{X\setminus E} |f|\,d\mu < \eps.
\end{equation}
\end{enumerate}
\end{lemma}
\begin{remark}
Observe that whenever $\mu(X) < \infty$,  condition (b) is trivially satisfied by setting $E = X$, which is why it is omitted in literature that discusses the Dunford--Pettis theorem over spaces of finite measure.
\end{remark}

\section{The main result}
\label{uc}

In this section we will carefully record all of the notation and technical lemmata used in the poof of the main result, Theorem~\ref{main}, and then we will prove it.  The reader is reminded that we are always assuming $(X,d,\mu)$ is a metric-measure space, where $\mu$ is a Borel regular doubling measure. However, unless explicitly stated, we do not assume that the space supports a $p$-Poincar\'e inequality.

\subsection{Notation} 
For each $k\in\bbbz$, let $\{B_i\}_{i=1}^{M(k)}$, $M(k)\in\mathbb{N}\cup\{\infty\}$, be a covering of $X$ by balls of radius $2^{-k}$ such that the balls in the family  $\{\frac{1}{5}B_i\}_{i=1}^{M(k)}$ are pairwise disjoint. The existence of such coverings follows from the familiar $5r$-covering lemma.

The doubling property of the measure $\mu$ implies that for each fixed $\theta\in[1,\infty)$, the family $\{\theta B_i\}_{i=1}^{M(k)}$ of enlarged balls has  bounded overlapping, in the sense that there exists a constant $C_0\in[1,\infty)$ such that $\sum_i\chi_{\theta B_i}(x)\leq C_0$ for every $x\in X$. 
Note that $C_0$ depends only on $\theta$ and $C_d$ (the doubling constant of $\mu$). In particular, $C_0$ is independent of $k$. 

For each $k\in\bbbz$, we have a different family of balls (referred to as \textit{balls of generation $k$}) and we will write $B^k_i:=B_i$ if we wish to stress for which $k\in\bbbz$ the family was constructed.

We say that  balls $B^k_i$ and $B^k_j$ are \textit{neighbors} if $\operatorname{dist}(B^k_i,B^k_j)<2^{-k}$, and we will write $B^k_i\sim B^k_j$ in this case. Note that there exists $N\in\bbbn$, such that each ball has at most $N$ neighbors, where $N$ depends only on the doubling constant of $\mu$ and, in particular, is independent of $k$. 

If $B_{i,1},\dots,B_{i,{n_i}}$, $n_i< N$ are \textit{all} of the neighbors of $B_i$ then we set
$$
B_{i,{n_i+1}},\dots,B_{i,{N}}:=B_i.
$$
That is, we set the last $N-n_i$ balls in the sequence $\{B_{i,j}\}_{j=1}^N$ to be identical copies of $B_i$. While this construction is somewhat formal, for reasons that will be clear later, we need to have the same number of balls ``around'' each of the $B_i$'s. 

Let $A_1:=B_1$ and $A_i:=B_i\setminus(B_1\cup\cdots\cup B_{i-1})$ for each $i\geq2$. Then $X=\bigcup_{i=1}^{M(k)}A_i$, and so, in particular, $\{A_i\}_{i=1}^{M(k)}$ is a partition of $X$ into pairwise disjoint sets. For each $x\in X$, there is a unique $i$ such that $x\in A_i$. In other words, $i$ is the smallest index such that $x\in B_i$. As such, we define $B[x]:=B_i$ and set $B[x,j]:=B_{i,j}$ for $j\in[1,N]$. In particular, $B[x,j]=B[x]$ if $j\in(n_i,N]$.

For $u\in L^1_\loc(X)$ and $k\in\bbbz$, we define
\begin{equation}
\label{ukdef}
S_ku:=
\sum_{i=1}^{M(k)}u_{B_i}\,\chi_{A_i},
\end{equation}
and note that  $S_ku(x)=u_{B[x]}$ for each $x\in X$. According to Lebesgue's differentiation theorem (Lemma~\ref{T5}), $S_ku\to u$ pointwise $\mu$-a.e. in $X$ as $k\to\infty$. 

For $u\in L^1_\loc(X)$ and $k\in\bbbz$, we also define
\begin{equation*}
\label{Tkdef}
T_ku(x):=2^k\big[u_{B[x]}-u_{B[x,1]},\dots,u_{B[x]}-u_{B[x,N]}\big]\in\bbbr^N,
\end{equation*}
for $x\in X$, or equivalently,
$$
T_ku:=2^k\sum_{i=1}^{M(k)}\big[u_{B_i}-u_{B_{i,1}},\dots,u_{B_i}-u_{B_{i,N}}\big]\chi_{A_i}.
$$
Observe that if $x\in A_i$ and $n_i<N$ (the actual number of neighbors of $B[x]=B_i$), then the vector $T_ku(x)\in\bbbr^N$ has zeros in the last $N-n_i$ components, i.e.,
$$
T_ku(x):=2^k\big[u_{B_i}-u_{B_{i,1}},\dots,u_{B_i}-u_{B_{i,{n_i}}},0,\dots,0\big].
$$
Equipping $\bbbr^N$ with the $\ell^p_N$ norm, $p\in[1,\infty)$, we have
\begin{equation}
\label{Tkdef2}
|T_ku(x)|_p=2^k\Bigg(\sum_{j=1}^N|u_{B[x]}-u_{B[x,j]}|^p\Bigg)^{1/p}=2^k\sum_{i=1}^{M(k)}\Bigg(\sum_{j=1}^N|u_{B_i}-u_{B_{i,j}}|^p\Bigg)^{1/p}\chi_{A_i}(x).
\end{equation}
In particular,
$$
|T_ku(x)|:=|T_ku(x)|_1=2^k\sum_{j=1}^N|u_{B[x]}-u_{B[x,j]}|=2^k\sum_{i=1}^{M(k)}\sum_{j=1}^N|u_{B_i}-u_{B_{i,j}}|\,\chi_{A_i}(x).
$$
Clearly the norms $|\cdot|$ and $|\cdot|_p$ are equivalent on $\bbbr^N$, but we will have to work with the norm $|\cdot|_p$ in order to prove uniform convexity of $N^{1,p}$. Note that when $\bbbr^N$ is equipped with $|\cdot|_p$, the $L^p(X,\ell^p_N)$ norm of $T_ku$ is
$$
\big\Vert|T_ku|_p\big\Vert_{L^p(X)}=2^k\Bigg(\sum_{i=1}^{M(k)}\sum_{j=1}^N\Big(|u_{B_i}-u_{B_{i,j}}|\mu(A_i)^{1/p}\Big)^p\Bigg)^{1/p},
$$
where we have used the fact that the $A_i$'s are pairwise disjoint. 

Fix $p \in [1, \infty)$ and define  
\begin{equation}
\label{newnorm}
\Vert u\Vert_{1,p}^* :=\Big(\Vert u\Vert_{L^p(X)}^p+\limsup_{k\to\infty} \big\Vert|T_{k} u|_p\big\Vert_{L^p(X)}^p\Big)^{1/p} \quad\mbox{for each $u \in  N^{1,p}(X)$.}
\end{equation}

Observe that components of $T_ku(x)$ are averaged difference quotients of $u$ in all possible directions, i.e., over all balls that are neighbours of $B[x]$. As we shall see, in some sense $|T_ku|$ (or $|T_ku|_p$) is an approximation of the minimal $p$-weak upper gradient of $u$ (see Lemma~\ref{almostug}, Proposition~\ref{gradcomparable}, and Theorem~\ref{ptwisethm}).

\subsection{Auxiliary results}
In this subsection we will prove technical lemmata which will be needed in the proof of the main result (Theorem~\ref{main}).

\begin{lemma}
\label{TkBDD}
Suppose that the space supports a $p$-Poincar\'e inequality \eqref{poincare} for some $p\in[1,\infty)$. Then there exists a constant $C=C(p,C_d,c_{\PI})>0$ such that 
if $u\in L^1_{\rm loc}(X)$ is Borel measurable and $g$ is a $p$-weak upper gradient of $u$ that is finite $\mu$-a.e., then for each $k\in\bbbz$, we have
\begin{equation}
\label{tkpwest}
|T_ku(x)|_p\leq 
C\Bigg(\,\mvint_{5\lambda B[x]} g^p\,d\mu\Bigg)^{1/p}
\quad\mbox{for all $x\in X$.}
\end{equation}
Therefore, there is a constant $C'=C'(p,C_d,C_{\PI},\lambda)>0$ such that
\begin{equation}
\label{tkpwest-2}
\big\Vert|T_ku|_p\big\Vert_{L^p(X)}\leq C'\Vert g\Vert_{L^p(X)}.
\end{equation}
Consequently, if $u\in N^{1,p}(X)$, then the sequence $\{T_ku\}_{k\in\bbbz}$ is bounded in $L^p(X,\ell^p_N)$.
\end{lemma}

\begin{proof}
Fix $k\in\bbbz$ and $x\in X$, along with $j\in[1,N]$. Since $B[x]$ and $B[x,j]$ are neighbors (by definition) we have that $B[x,j]\subseteq 5 B[x]\subseteq 10 B[x,j]$. Therefore, the doubling condition of $\mu$ implies that $\mu(B[x,j]) \approx \mu(5B[x])$.
Applying Lemma~\ref{pwkpoin} to the pair
$(u,g)$ we can estimate
\begin{equation*}
\begin{split}
|u_{B[x]} - u_{B[x,j]}| 
&\leq
|u_{B[x]} - u_{5B[x]}|+|u_{5B[x]} - u_{B[x,j]}|\\
&\lesssim 
\mvint_{5B[x]} |u-u_{5B[x]}|\,d\mu
\leq C 2^{-k} \Bigg(\,\,\mvint_{5\lambda B[x]} g^p\,d\mu \Bigg)^{1/p},
\end{split}
\end{equation*}
where $C\in(0,\infty)$ depends only on $C_d$ and $c_{\PI}$. From the formula for $|T_ku|_p$, we have
\begin{equation}
\label{eq5}
|T_ku(x)|_p
\leq 2^{k}\Bigg(\sum_{j=1}^NC^p2^{-kp}\mvint_{5\lambda B[x]} g^p\,d\mu\Bigg)^{1/p}
=C \cdot N^{1/p}\Bigg(\,\mvint_{5\lambda B[x]} g^p\,d\mu\Bigg)^{1/p},
\end{equation}
where $C$ and $N$ only depend on $C_d$ and $c_{\PI}$.
This proves \eqref{tkpwest}.

Turning our attention to proving \eqref{tkpwest-2}, observe that estimate \eqref{eq5} is equivalent to
\begin{equation}
\label{eq6}
|T_ku(x)|_p\leq 
C\cdot N^{1/p}\Bigg(\sum_{i=1}^{M(k)}\chi_{A_i}(x)\mvint_{5\lambda B_i} g^p\,d\mu\Bigg)^{1/p},
\end{equation}
because the right hand sides of \eqref{eq5} and \eqref{eq6} are equal.
The bounded overlapping of the family of enlarged balls $\{5\lambda B_i\}_{i=1}^{M(k)}$ and the fact that $A_i\subseteq B_i$ together yield
\[
\big\Vert|T_ku|_p\big\Vert_{L^p(X)}^p\leq 
C^pN \sum_{i=1}^{M(k)}\mu(A_i) \Bigg(\,\,\mvint_{5\lambda B_i} g^p\,d\mu\Bigg)\leq
C^pN \sum_{i=1}^{M(k)} \,\,\int_{5\lambda B_i} g^p\,d\mu \leq
C'\Vert g\Vert_{L^p(X)}^p,
\]
where $C'=C'(p,C_d,C_{\PI},\lambda)$.

Finally, if $u\in N^{1,p}(X)$ then \eqref{tkpwest-2} applied with $g=g_u\in L^p(X)$ proves boundedness of 
$\{T_ku\}_{k\in\bbbz}$ in $L^p(X,\ell^p_N)$.
\end{proof}

The next result follows immediately from the definition of $\Vert\cdot\Vert_{1,p}^\ast$ in \eqref{newnorm}, Lemma~\ref{T3}, and Lemma~\ref{TkBDD}.

\begin{corollary}
\label{objectprops}
Suppose that the space supports a $p$-Poincar\'e inequality for some $p\in[1,\infty)$. Then $\Vert\cdot\Vert_{1,p}^*\colon N^{1,p}(X)\to [0, \infty)$ as in \eqref{newnorm} is a well-defined norm 
on $N^{1,p}(X)$ and there exists $C\in(0,\infty)$ satisfying
$$
\Vert u\Vert_{1,p}^*\leq C\|u\|_{N^{1,p}(X)}\quad\mbox{for all $u\in N^{1,p}(X)$.}
$$
\end{corollary}

The reader is reminded of the definition of $S_ku$ in \eqref{ukdef}.
\begin{lemma}
\label{pro:T_ku-converges-balls}
Let $u \in N^{1,p}(X)$ with $p \in [1, \infty)$ and assume that $\{h_k\}_{k=1}^\infty$ is a sequence of nonnegative Borel functions in $L^p(X)$ such  that 
\begin{equation}
\label{bwr-249}
|S_ku(x)-S_ku(y)|\leq \int_\gamma h_k\,ds,
\end{equation}
whenever $k\in\bbbn$, $x,y\in X$ satisfy $d(x,y)\geq 2^{-k}$, and $\gamma$ is a rectifiable curve connecting $x$ and $y$. If  $\{h_k\}_{k=1}^\infty$ contains a subsequence that converges weakly in $L^p(X)$ to some nonnegative Borel function $h \in L^p(X)$, then $h$ is a $p$-weak upper gradient of $u$ and hence, $h\ge g_u$ pointwise $\mu$-a.e.\@ in $X$, where $g_u \in L^p(X)$ is the minimal $p$-weak upper gradient of $u$. 
\end{lemma}

\begin{proof}
Without loss of generality, we can assume that $\{h_k\}_{k=1}^\infty$ converges weakly in $L^p(X)$ to some Borel function $h \in L^p(X)$. Then by Mazur's lemma (see, e.g., \cite[p.~19]{HKST}), there exists a sequence
\[
g_k := \sum_{\ell=k}^{L(k)} \alpha_{k,\ell} h_\ell \to h\,\,\text{ in $L^p(X)$ as }k\to\infty,
\]
where $\alpha_{k,\ell}\geq0$ and $\sum_{\ell=k}^{L(k)} \alpha_{k,\ell} = 1$ for each $k\in\bbbn$ (with $L(k)\in\bbbn$). By further passing to a subsequence, if necessary, we can assume that $g_k\to h$ pointwise $\mu$-a.e. in $X$ as $k\to\infty$.  Consider the corresponding family of convex combinations of $S_ku$, $v_k := \sum_{\ell=k}^{L(k)} \alpha_{k,
\ell}S_\ell u$. Since $S_ku \to u$ pointwise $\mu$-a.e. in $X$ by Lebesgue's differentiation theorem (Lemma~\ref{T5}), we have that $v_k \to u$ pointwise $\mu$-a.e.\@ in $X$, as well. 

Clearly, if $d(x,y)\geq 2^{-k}$ for some $k\in\bbbn$ and $\gamma\in\Gamma(X)$ connects $x$ and $y$, then by \eqref{bwr-249} we have
\begin{equation}
\label{eq14}
|v_k(x)-v_k(y)|\leq \int_\gamma g_k\, ds.
\end{equation}
Define $\tilde{u}\colon X\to[-\infty,\infty]$ by setting $\tilde{u}(x) := \limsup_{k\to\infty} v_k(x)$ for every $x\in X$ and note that $\tilde{u}=u$ pointwise $\mu$-a.e. in $X$. We will prove that $\tilde{u}$ is finite everywhere on the image $|\gamma|$ for ${\rm Mod}_p$-a.e curve $\gamma\in\Gamma(X)$. To this end, by Fuglede's lemma (Lemma~\ref{fuglede}), there is a set $\Gamma_1\subseteq\Gamma(X)$ with ${\rm Mod}_p(\Gamma_1)=0$ and a subsequence of $\{g_k\}_{k=1}^\infty$ (also denoted by $\{g_k\}_{k=1}^\infty$) such that
\begin{equation}
\label{rwq43}
\int_\gamma g_k\,ds \to \int_\gamma h\,ds \,\in\bbbr
\quad
\mbox{as $k\to\infty$,}
\end{equation}
for every curve $\gamma\in\Gamma(X)\setminus\Gamma_1$. Next, let $E$ be the set of all $x\in X$ for which the convergence $v_k(x)\to u(x)\in\bbbr$ does not hold, and set
$$
\Gamma_2:=\big\{\gamma\in\Gamma(X):\ |\gamma|\subseteq E\big\}.
$$
Note that $E\subseteq X$ is $\mu$-measurable and $\mu(E)=0$, which implies that $\|\infty\cdot\chi_E\|_{L^p(X)}=0$. This, together with the observation that $\infty\cdot\chi_E\in F(\Gamma_2)$, immediately gives ${\rm Mod}_p(\Gamma_2)=0$ and hence, ${\rm Mod}_p(\Gamma_1\cup\Gamma_2)=0$. 

Now fix a curve $\gamma\in\Gamma(X)\setminus(\Gamma_1\cup\Gamma_2)$
and let $\widetilde{\gamma}$ be the arc-length parameterization of $\gamma$. We claim that the sequence $\{v_k(\widetilde{\gamma}(s))\}_{k=1}^\infty$ of real numbers is bounded for every $s\in[0,\ell(\gamma)]$. Let $s\in[0,\ell(\gamma)]$ and note that since $\gamma\not\in\Gamma_2$,  there is a point $t\in[0,\ell(\gamma)]$ such that $\widetilde{\gamma}(t)\not\in E$. By definition of the set $E$, we have that $v_k(\widetilde{\gamma}(t))\to u(\widetilde{\gamma}(t))\in\bbbr$. In particular, $\{v_k(\widetilde{\gamma}(t))\}_{k=1}^\infty$ is a bounded sequence. To proceed, it is enough to consider the scenario when $s\leq t$ as the other case is handled similarly. If $\widetilde{\gamma}(s)=\widetilde{\gamma}(t)$ then $\{v_k(\widetilde{\gamma}(s))\}_{k=1}^\infty$ is bounded by the choice of $t$. If, on the other hand, $\widetilde{\gamma}(s)\neq\widetilde{\gamma}(t)$ then we have $d(\widetilde{\gamma}(t),\widetilde{\gamma}(s))\geq2^{-k}$ for all sufficiently large $k\in\mathbb{N}$ and so, by appealing to \eqref{eq14}  we can write
\begin{equation}
\label{rqi-47}
\begin{split}
\big|v_k(\widetilde{\gamma}(s))\big|
&\leq
\big|v_k(\widetilde{\gamma}(s))-v_k(\widetilde{\gamma}(t))\big|+\big|v_k(\widetilde{\gamma}(t))\big|
\\
&\leq\int_s^tg_k(\widetilde{\gamma}(\tau))d\tau+\big|v_k(\widetilde{\gamma}(t))\big|
\leq\int_\gamma g_k\,ds+\big|v_k(\widetilde{\gamma}(t))\big|.
\end{split}
\end{equation}
Since $\gamma\not\in\Gamma_1$, we have that $\int_\gamma g_k\,ds$ converges to the finite number \eqref{rwq43}, as $k\to\infty$ and hence, is bounded. Therefore, the right-hand side of \eqref{rqi-47} is bounded by a finite constant that is independent of $k$, and it follows that $\{v_k(\widetilde{\gamma}(s))\}_{k=1}^\infty$ is a bounded sequence for each fixed $s\in[0,\ell(\gamma)]$. Consequently, $\tilde{u}$ is finite on the image $|\gamma|$ whenever $\gamma\in\Gamma(X)\setminus(\Gamma_1\cup\Gamma_2)$.

Moving on, we claim next that $h$ is a $p$-weak upper gradient of $\tilde{u}$. Fix $\gamma\in\Gamma(X)\setminus(\Gamma_1\cup\Gamma_2)$ and let $x,y \in X$ be the end-points of $\gamma$. If $x=y$, then the inequality $|\tilde{u}(x) - \tilde{u}(y)| \leq \int_\gamma h\,ds$ is trivially satisfied, since $\tilde{u}(x)=\tilde{u}(y)\in\mathbb{R}$. If $x \neq y$, then $d(x,y)\geq 2^{-k}$ for all $k\in\mathbb{N}$, large enough, and so  \eqref{eq14} is satisfied.  Since $\gamma\not\in\Gamma_1\cup\Gamma_2$, we have that \eqref{rwq43} holds and $\tilde{u}(x),\tilde{u}(y)\in\mathbb{R}$. As such, we can estimate
\[
  |\tilde{u}(x) - \tilde{u}(y)|\leq\limsup_{k\to\infty} |v_k(x) - v_k(y)| \leq \limsup_{k\to\infty} \int_\gamma g_k\,ds = \int_\gamma h\,ds\,.
\]
Therefore, $h \in L^p(X)$ is a $p$-weak upper gradient of $\tilde{u}$, and hence $\tilde{u} \in N^{1,p}(X)$. Since $u=\tilde{u}$ pointwise $\mu$-a.e. in $X$ and both $u$ and $\tilde{u}$ belong to $N^{1,p}(X)$, the function $h$ is also a $p$-weak upper gradient of $u$ by Lemma~\ref{T3}. Therefore, $h \ge g_u$ pointwise $\mu$-a.e.\@ in $X$ by the definition of a minimal $p$-weak upper gradient. This completes the proof of Lemma~\ref{pro:T_ku-converges-balls}.
\end{proof}

We will show that the sequence $h_k:=4|T_k u|_p$ satisfies the hypotheses of Lemma~\ref{pro:T_ku-converges-balls}. We first verify estimate \eqref{bwr-249}.

\begin{lemma}
\label{almostug}
Let  $u\in L^{1}_\loc(X)$ and suppose that $\gamma$ is a rectifiable curve in $X$ with endpoints $x$ and $y$. If $d(x,y)\geq 2^{-k}$ for some $k\in\bbbz$, then
$$
|S_ku(x)-S_ku(y)|\leq 4\int_\gamma |T_ku|_p\,ds,
$$
where $S_ku$ is as in \eqref{ukdef}.
\end{lemma}
\begin{proof}
We can assume that $\gamma:[0,L]\to X$ is parametrized by arc-length and $x=\gamma(0)$, $y=\gamma(L)$.
Consider a partition
$$
0=t_0<t_1<\cdots<t_m=L,
$$
of $[0,L]$ such that the length of each subinterval $[t_{i-1},t_i]$ satisfies
$$
2^{-(k+1)}\leq t_i-t_{i-1}<2^{-k}. 
$$
This is possible because $L=\ell(\gamma)\geq d(x,y)\geq 2^{-k}$. Since $\gamma$ is parametrized by the arc-length, it follows that
\begin{equation}
\label{len-35}
\ell(\gamma\vert_{[t_{i-1},t_i]})=t_i-t_{i-1}< 2^{-k}.
\end{equation}
Observe that
\begin{equation}
\label{ynap-345}
\begin{split}
|S_ku(x)-S_ku(y)|&\leq\sum_{i=1}^m\big|S_ku(\gamma(t_i))-S_ku(\gamma(t_{i-1}))\big|\\
&=\sum_{i=1}^m\big|u_{B[\gamma(t_i)]}-u_{B[\gamma(t_{i-1})]}\big|.
\end{split}
\end{equation}
On the other hand, if $i\in\{1,\dots,m\}$ is fixed then for any $t\in[t_{i-1},t_i]$ we have
\begin{equation}
\label{eni-3u}
\big|u_{B[\gamma(t_i)]}-u_{B[\gamma(t_{i-1})]}\big|\leq
\big|u_{B[\gamma(t_i)]}-u_{B[\gamma(t)]}\big|+\big|u_{B[\gamma(t)]}-u_{B[\gamma(t_{i-1})]}\big|.
\end{equation}
It follows from \eqref{len-35} that the distance between $\gamma(t)$ and any of the points $\gamma(t_{i-1})$ and $\gamma(t_{i})$ is less than $2^{-k}$ and hence, each of the balls $B[\gamma(t_{i-1})]$ and $B[\gamma(t_{i})]$ is a neighbor of $B[\gamma(t)]$. Combining this with \eqref{eni-3u} and the formula for $|T_ku(\gamma(t))|_p$ in \eqref{Tkdef2} yields
$$
\big|u_{B[\gamma(t_i)]}-u_{B[\gamma(t_{i-1})]}\big|\leq 2\cdot 2^{-k}|T_ku(\gamma(t))|_p.
$$
Therefore, integration with respect to $t\in[t_{i-1},t_i]$ gives
\begin{align*}
(t_i-t_{i-1})\big|u_{B[\gamma(t_i)]}-u_{B[\gamma(t_{i-1})]}\big|&\leq
2\cdot 2^{-k}\int_{t_{i-1}}^{t_i}|T_ku(\gamma(t))|_p\,dt.
\end{align*}
Since $t_i-t_{i-1}\geq 2^{-(k+1)}$ we have
$$
\big|u_{B[\gamma(t_i)]}-u_{B[\gamma(t_{i-1})]}\big|\leq 4\int_{t_{i-1}}^{t_i}|T_ku(\gamma(t))|_p\,dt.
$$
Given that $i\in\{1,\dots,m\}$ was arbitrary, we can add the inequalities in \eqref{ynap-345} to obtain
\[
|S_ku(x)-S_ku(y)|\leq 4\int_{0}^{L}|T_ku(\gamma(t))|_p\,dt=4\int_{\gamma}|T_ku|_p\,ds.\qedhere
\]
\end{proof}

Next we show that if the space supports a $p$-Poincar\'e inequality for some $p\in [1,\infty)$, and $u\in N^{1,p}(X)$, then we can always extract a subsequence of $\{|T_k u|_p\}_{k=1}^\infty$ that converges weakly in $L^p(X)$. In the case when $p>1$, we can rely on the reflexivity of $L^p$ and Lemma~\ref{TkBDD}. However, the case of $p=1$ is more delicate; it relies on the
Dunford--Pettis theorem (Lemma~\ref{thm:DunfordPettis}) and some ideas from \cite{FraHajKos}.

For the next result, see also \cite[Lemma~6]{FraHajKos}. We will only need it for $p=1$.
\begin{lemma}
\label{lem:Tku-equiint-balls}
Suppose that the space supports a $p$-Poincar\'e inequality for some $p\in [1,\infty)$. If $u\in N^{1,p}(X)$, then every subsequence of $\{|T_ku|_p^p \}_{k=1}^\infty$ has a further subsequence that is weakly convergent in $L^1(X)$.
\end{lemma}
\begin{proof}
Fix $u\in N^{1,p}(X)$.
We will prove that $\{|T_ku|_p^p \}_{k=1}^\infty$ satisfies (a) and (b) in Lemma~\ref{thm:DunfordPettis}.

To verify (b), fix $\eps\in(0,\infty)$ and $k\in\bbbn$. Since inequality \eqref{tkpwest} is satisfied with $g=g_u\in L^p(X)$, we have
$$
|T_k u(x)|_p^p \lesssim 
\mvint_{5\lambda B[x]} g_u^p\,d\mu=
\sum_{i=1}^{M(k)}\chi_{A_i}(x)\,\mvint_{5\lambda B_i} g_u^p\,d\mu
\quad\mbox{for every $x \in X$.}
$$
Consequently, since $A_i \subseteq B_i$, for every measurable set $S\subseteq X$, we have
\begin{equation}
\label{eq8}
\int_S |T_k u|_p^p\,d\mu \lesssim \sum_{i=1}^{M(k)} \frac{\mu(S \cap B_i)}{\mu(5\lambda B_i)} \int_{5\lambda B_i} g_u^p\,d\mu
\leq\sum_{i:\, S\cap B_i\neq\varnothing}\ \int_{5\lambda B_i} g_u^p\, d\mu.
\end{equation}

Fix $x_o\in X$, $R>6\lambda$, and let $S_R:=X\setminus B(x_o,2R)$.
Each of the balls $B_i$ has radius $2^{-k}<1\leq\lambda$. Thus, if $S_R\cap B_i\neq\varnothing$, then 
$5\lambda B_i\cap B(x_o,R)=\varnothing$, by a simple application of the triangle inequality, and hence \eqref{eq8} and the bounded overlapping of the balls $\{ 5\lambda B_i\}_i$ yield
$$
\int_{X\setminus B(x_o,2R)} |T_ku|_p^p\, d\mu\lesssim \int_{X\setminus B(x_o,R)} g_u^p\, d\mu<\eps,
$$
provided $R$ is sufficiently large. This proves condition (b) in Lemma~\ref{thm:DunfordPettis} with $E:=B(x_o,2R)$.

Next, we prove that condition (a) holds.
Note that Lemma~\ref{TkBDD} implies that $\{|T_ku|_p^p \}_{k=1}^\infty$ is bounded in $L^1(X)$.
Thus it remains to prove that the family is equi-integrable. 

Fix $\eps\in(0,\infty)$ and $k\in\bbbn$, and let $\sigma\in(0,\infty)$ be any number. The value of $\sigma$ will be fixed later.

Given a $\mu$-measurable set $S \subseteq X$, we define $\mathcal{G}$ to be the collection of all integers $i\in[1,M(k)]$ satisfying $\mu(S \cap B_i) \le \sigma\mu(5\lambda B_i)$, and we let $\mathcal{B}$ consist of all integers $i\in[1,M(k)]\setminus\mathcal{G}$. 
Note that $\mu(5\lambda B_i)<\mu(S\cap B_i)/\sigma$ for all $i\in \mathcal{B}$.
Thus $\mathcal{G}$ and $\mathcal{B}$ partition the set of integers in $[1,M(k)]$ and \eqref{eq8} yields
\begin{equation}
\label{eq9}
\int_S |T_k u|_p^p\,d\mu \leq C_1\Bigg(
\sigma\sum_{i\in\mathcal{G}}\,\int_{5\lambda B_i} g_u^p\, d\mu +
\sum_{i\in\mathcal{B}}\, \int_{5\lambda B_i} g_u^p\, d\mu \Bigg)\, ,
\end{equation}
where the constant $C_1$ does not depend on $\sigma$ or $k$.

Assume that the overlapping constant of the balls $\{5\lambda B_i\}_i$ is bounded by $C_2$. Now we fix $\sigma\in (0,\infty)$ such that
$$
C_2\sigma\Vert g_u\Vert_p^p<\frac{\eps}{2C_1}\, .
$$
Then the first sum in \eqref{eq9} can be estimated by
\begin{equation}
\label{eq10}
\sigma\sum_{i\in\mathcal{G}}\,\int_{5\lambda B_i} g_u^p\, d\mu\leq C_2\sigma\int_X g_u^p\, d\mu<\frac{\eps}{2C_1}\, .
\end{equation}
Regarding the second sum in \eqref{eq9}, we have
\begin{equation}
\label{eq11}
\sum_{i\in\mathcal{B}}\, \int_{5\lambda B_i} g_u^p\, d\mu\leq C_2 \int_G g_u^p\, d\mu
\quad
\text{where}
\quad
G:=\bigcup_{i\in\mathcal{B}} 5\lambda B_i.
\end{equation}
Note that
\begin{equation}
\label{eq12}
\mu(G)\le\sum_{i\in\mathcal{B}}\mu(5\lambda B_i)\leq \sum_{i\in\mathcal{B}} \frac{\mu(S \cap B_i)}{\sigma} \le \frac{C_2\,\mu(S)}{\sigma}\,.
\end{equation}
Absolute continuity of the integral yields $\tilde{\delta}\in (0,\infty)$ such that
\begin{equation}
\label{eq13}
\int_G g_u^p\, d\mu < \frac{\eps}{2C_1C_2}\, ,
\end{equation}
provided $\mu(G)<\tilde{\delta}$.

Let $\delta:=\sigma\tilde{\delta}/C_2$. If $\mu(S)<\delta$, then $\mu(G)<\tilde{\delta}$ by \eqref{eq12} and hence \eqref{eq13} is satisfied. This, in concert with \eqref{eq9}, \eqref{eq10}, and \eqref{eq11} yield
$$
\int_S |T_ku|_p^p\, d\mu<
C_1\Big(\frac{\eps}{2C_1}+C_2\, \frac{\eps}{2 C_1C_2}\Big)=\eps,
$$
and that completes the proof of the equi-integrability and the proof of Lemma~\ref{lem:Tku-equiint-balls}.
\end{proof}

\begin{corollary}
\label{cor:subseq-T_k-wkconv}
Suppose the space supports a $p$-Poincar\'e inequality for some $p\in [1,\infty)$. If $u \in N^{1,p}(X)$ then, every subsequence of $\{|T_k u|_p\}_{k=1}^\infty$ has a further subsequence that converges weakly in $L^p(X)$.
\end{corollary}

\begin{proof}
If $p>1$, then the sequence $\{|T_k u|_p\}_{k=1}^\infty$ is bounded in $L^p(X)$ by Lemma~\ref{TkBDD}, and
the result follows from the reflexivity of $L^p(X)$.
If $p=1$, then the existence of a weakly convergent subsequence is guaranteed by Lemma~\ref{lem:Tku-equiint-balls}.
\end{proof}

\subsection{Proof of the main result}
\label{ssec:proof}
\begin{proof}[Proof of Theorem~\ref{main}]
We need to prove that:
\begin{itemize}[label={\footnotesize\textbullet}]
\item $\Vert\cdot\Vert_{1,p}^*\approx\Vert\cdot\Vert_{N^{1,p}(X)}$ on $N^{1,p}(X)$ when $p\in [1,\infty)$;
\vskip.08in

\item the norm $\Vert\cdot\Vert_{1,p}^*$ is uniformly convex on $N^{1,p}(X)$ when $p\in (1,\infty)$.
\end{itemize}
Then reflexivity of $N^{1,p}(X)$, $p\in (1,\infty)$, will follow directly from the Milman-Pettis theorem, Lemma~\ref{milmanpettis}. 

Therefore, the proof of Theorem~\ref{main} is contained in Proposition~\ref{gradcomparable} and Proposition~\ref{F-UniCon} below.

\begin{proposition}
\label{gradcomparable}
Suppose the space supports a $p$-Poincar\'e inequality for some $p\in [1,\infty)$.
Then there exists  $C=C(p,C_d,c_{\PI},\lambda)\in(0,\infty)$ such that
$$
4^{-1}\|g_u\|_{L^p(X)}\leq \limsup_{k\to\infty} \big\Vert|T_{k} u|_p\big\Vert_{L^p(X)}\leq C\|g_u\|_{L^p(X)},
$$
for all $u\in N^{1,p}(X)$. Consequently,  $\Vert u\Vert_{1,p}^*\approx\Vert u\Vert_{N^{1,p}(X)}$ for all $u\in N^{1,p}(X)$.
\end{proposition}

\begin{proof}
Fix $u \in N^{1,p}(X)$ and let $g_u \in L^p(X)$ denote the minimal $p$-weak upper gradient of $u$. In view of Lemma~\ref{TkBDD}, we immediately have that
\begin{equation}
\label{eq4}
\limsup_{k\to\infty} \big\Vert|T_{k} u|_p\big\Vert_{L^p(X)}\leq C\|g_u\|_{L^p(X)}
\end{equation}
for some $C=C(p,C_d,c_{\PI},\lambda)\in(0,\infty)$. 

To see the opposite inequality, take a subsequence $\{|T_{k_j} u|_p\}_{j=1}^\infty$ of $\{|T_{k} u|_p\}_{k=1}^\infty$ such that
\begin{equation}
\label{liminf-est}
\lim_{j\to\infty} \big\Vert|T_{k_j} u|_p\big\Vert_{L^p(X)}
=\liminf_{k\to\infty} \big\Vert|T_{k} u|_p\big\Vert_{L^p(X)}.
\end{equation}
In light of Corollary~\ref{cor:subseq-T_k-wkconv}, by passing to a further subsequence, we can assume $\{|T_{k_j} u|_p\}_{j=1}^\infty$ converges weakly in $L^p(X)$.
Let $|T|(u) \in L^p(X)$ be a Borel representative of the weak limit of $\{|T_{k_j} u|_p\}_{j=1}^\infty$  and set $h_k:=4|T_{k} u|_p$ and $h:=4|T|(u)$. Note that $h$ and each $h_k$ are nonnegative  Borel functions. Since $\{h_{k_j}\}_{j=1}^\infty$ converges weakly to $h$ in $L^p(X)$, by appealing to Lemma~\ref{almostug}, we can conclude that the pair $\big(\{h_k\}_{k=1}^\infty, h\big)$ satisfies the hypotheses of Lemma~\ref{pro:T_ku-converges-balls}. Therefore, we have that $h$ is a $p$-weak upper gradient of $u$ and hence, $h\geq g_u$ pointwise $\mu$-a.e. in $X$. Combining this fact with 
\eqref{liminf-est} and the lower semicontinuity of the $L^p$-norm (with respect to the weak convergence), we can estimate
\begin{equation}
\label{svv-2}
\begin{split}
\limsup_{k\to\infty} \big\Vert|T_{k} u|_p\big\Vert_{L^p(X)} 
&\ge 
\liminf_{k\to\infty} \big\Vert |T_k u|_p\big\Vert_{L^p(X)} 
= 4^{-1}\lim_{j\to\infty}\Vert h_{k_j}\Vert_{L^p(X)}\\
&\ge 
4^{-1}\|h\|_{L^p(X)} \ge 4^{-1}\|g_u\|_{L^p(X)}.
\end{split}
\end{equation}
The proof of Proposition~\ref{gradcomparable} is now complete.
\end{proof}

\begin{remark}
\label{liminf-limsup-equiv}
Combining \eqref{eq4} and \eqref{svv-2} we can conclude that for $p\in[1,\infty)$, there is a finite constant $\xi=\xi(p,C_d,c_{\PI},\lambda)\ge1$ satisfying
\begin{equation*}
\label{eq:liminf-limsup-equiv}
\liminf_{k\to\infty}\big\Vert|T_{k} u|_p\big\Vert_{L^p(X)}\le \limsup_{k\to\infty} \big\Vert|T_{k} u|_p\big\Vert_{L^p(X)} \le \xi\liminf_{k\to\infty} \big\Vert|T_{k} u|_p\big\Vert_{L^p(X)},
\end{equation*}
for every $u\in N^{1,p}(X)$. 
\end{remark}

We will now proceed to showing that $\Vert\cdot\Vert_{1,p}^*$ is uniformly convex on $N^{1,p}$ when $p\in(1,\infty)$.

\begin{proposition}
\label{F-UniCon}
Suppose the space supports a $p$-Poincar\'e inequality for some $p\in (1,\infty)$. Then the
norm $\Vert\cdot\Vert_{1,p}^*$ is uniformly convex  on $N^{1,p}(X)$.  In particular, the Banach space $(N^{1,p}(X),\|\cdot\|_{N^{1,p}(X)})$ is reflexive. 
\end{proposition}

\begin{proof}
Fix  $\eps\in(0,\infty)$. We will first prove that there exists $\delta\in(0,\infty)$ such that $\Vert u + v\Vert_{1,p}^*\leq 2(1-\delta)$ whenever $u, v \in N^{1,p}(X)$ satisfy $\Vert u\Vert_{1,p}^*<1$, $\Vert v\Vert_{1,p}^*<1$, and $\Vert u - v\Vert_{1,p}^*>\eps$. Fix $u, v \in N^{1,p}(X)$ as above. Then by definition of $\Vert\cdot\Vert_{1,p}^*$, we have that
\begin{equation}
\label{xwui-29}
\left(\|u\|_{L^p(X)}^p + \big\Vert|T_{k} u|_p\big\Vert_{L^p(X)}^p\right)^{1/p} <1
\qquad
\text{and}
\qquad\left(\|v\|_{L^p(X)}^p + \big\Vert|T_{k} v|_p\big\Vert_{L^p(X)}^p\right)^{1/p} < 1
\end{equation}
for all sufficiently large $k\in\bbbn$.  In light of Remark~\ref{liminf-limsup-equiv}, we can estimate  
\begin{align*}
\eps<\Vert u - v\Vert_{1,p}^* & \le \left(\|u-v\|_{L^p(X)}^p + \xi^p \liminf_{k\to\infty} \big\Vert|T_{k}(u-v)|_p\big\Vert_{L^p(X)}^p\right)^{1/p} \\
 & \le \xi \left(\|u-v\|_{L^p(X)}^p + \liminf_{k\to\infty} \big\Vert|T_{k}u-T_{k}v|_p\big\Vert_{L^p(X)}^p\right)^{1/p},
\end{align*}
where we have used the fact that $\xi\ge1$ and $T_k$ is linear in obtaining the last inequality. Consequently, \eqref{xwui-29} and
\begin{equation}
\label{xwui-30}
\big(\|u-v\|_{L^p(X)}^p + \big\Vert|T_{k}u -T_{k}v|_p\big\Vert_{L^p(X)}^p\big)^{1/p} > \eps/\xi,
\end{equation}
hold true for all sufficiently large $k\in\bbbn$. Fix such a $k$. Since $T_{k}u$ and $T_{k}v$ are vectors in $\mathbb{R}^N$, we can write
$$
\big\Vert|T_{k}u|_p\big\Vert_{L^p(X)}^p=\sum_{\ell=1}^N\big\Vert T^\ell_{k}u\big\Vert_{L^p(X)}^p
\quad
\mbox{and}
\quad
\big\Vert|T_{k}v|_p\big\Vert_{L^p(X)}^p=\sum_{\ell=1}^N\big\Vert T^\ell_{k}v\big\Vert_{L^p(X)}^p,
$$
where $T_{k}u=(T^1_{k}u,\dots,T^N_{k}u)$ and $T_{k}v=(T^1_{k}v,\dots,T^N_{k}v)$. 
Therefore, if we let
$$
f:=(u,T^1_{k}u,\dots,T^N_{k}u)\qquad\mbox{and}\qquad
g:=(v,T^1_{k}v,\dots,T^N_{k}v),
$$
then $f,g\in L^p(X,\ell^p_{N+1})$ and, with $\Phi$ defined as in \eqref{reoi}, a rewriting of \eqref{xwui-29} and \eqref{xwui-30} yields
\begin{equation*}
\begin{gathered}
\Phi(f)=\left(\|u\|_{L^p(X)}^p + \big\Vert|T_{k} u|_p\big\Vert_{L^p(X)}^p\right)^{1/p}<1,
\quad
\Phi(g)=\left(\|v\|_{L^p(X)}^p + \big\Vert|T_{k} v|_p\big\Vert_{L^p(X)}^p\right)^{1/p} < 1\\
\mbox{and}
\quad
\Phi(f-g)=\left(\|u-v\|_{L^p(X)}^p + \big\Vert|T_{k}u - T_{k}v|_p\big\Vert_{L^p(X)}^p\right)^{1/p}>\eps/\xi.
\end{gathered}
\end{equation*}
By Corollary~\ref{clarkson}, $L^p(X,\ell^p_{N+1})$ is uniformly convex and so (keeping in mind Lemma~\ref{unifconvex-clsd}) there exists $\delta\in(0,\infty)$, which depends on $\varepsilon$ and $\xi$, but is independent of $f$ and $g$ (in particular, $\delta$ is independent of $u$, $v$, and $k$), such that 
\begin{equation}
\label{gamq-58}
\left(\|u+v\|_{L^p(X)}^p + \big\Vert|T_{k} (u+v)|_p\big\Vert_{L^p(X)}^p\right)^{1/p} = \Phi(f+g) \le 2(1-\delta).
\end{equation}
Note that we have used the linearity of $T_{k}$ in obtaining the equality in \eqref{gamq-58}. Given that \eqref{gamq-58} holds for all sufficiently large $k\in\bbbn$, it follows that $\Vert u+v\Vert_{1,p}^* \le 2(1-\delta)$. 

To complete the proof of the proposition, suppose that $u, v \in N^{1,p}(X)$ are such that $\Vert u\Vert_{1,p}^*=\Vert v\Vert_{1,p}^*=1$, and $\Vert u - v\Vert_{1,p}^*>\eps$. Then, for all $\theta\in(0,1)$ sufficiently close to 1, we have that $\theta u, \theta v \in N^{1,p}(X)$ satisfy $\Vert\theta u\Vert_{1,p}^*,\Vert\theta v\Vert_{1,p}^*<1$ and $\Vert\theta u - \theta v\Vert_{1,p}^*>\eps$. As such,  we have $\Vert \theta u + \theta v\Vert_{1,p}^*\leq 2(1-\delta)$ by what has been established above. Since $\delta$ is independent of $\theta$, passing to the limit as $\theta\to1^-$ yields $\Vert u + v\Vert_{1,p}^*\leq 2(1-\delta)$. Given that $\eps\in(0,\infty)$ was arbitrary, it follows that $\Vert\cdot\Vert_{1,p}^*$ is a uniformly convex norm on $N^{1,p}(X)$.

Finally, the assertion that $(N^{1,p}(X),\|\cdot\|_{N^{1,p}(X)})$ is reflexive follows as an immediate consequence of the Milman--Pettis theorem (see Lemma~\ref{milmanpettis}) and the fact that a reflexive space remains reflexive for an equivalent norm. The proof of Proposition~\ref{F-UniCon} is now complete.
\end{proof} 
\noindent This completes the proof of Theorem~\ref{main}.
\end{proof}

\section{Separability from reflexivity}
\label{sepa}

In this section we will prove  separability of $N^{1,p}(X)$ for $p\in[1,\infty)$ (Theorem~\ref{apes}). In its proof we will employ a general result that provides  a mechanism for using reflexivity to establish separability, see Proposition~\ref{reflextosep}. Recall that we always assume that the measure on $X$ is doubling and Borel regular.

Throughout this section, all vector spaces are over the field of real numbers. Also, as a notational convention, if $S$ is a set of vectors then we let $\Span_{\mathbb{Q}}S$ and $\Span S$ denote the set of all finite linear combinations of vectors in $S$ with coefficients in $\bbbq$ and $\bbbr$, respectively.

\begin{proposition}
\label{reflextosep}
If $T\colon V\to W$ is a linear and bounded injective map of a reflexive Banach space $V$ into a separable normed space $W$, then $V$ is separable.
\end{proposition}

\begin{proof}
It suffices to prove that the unit ball $B\subseteq V$ is separable. Given that $W$ is separable, there is a set $\{v_k: k\in\mathbb{N}\}\subseteq B$ such that the set $\{T(v_k): k\in\mathbb{N}\}$ is dense in $T(B)$. Since $\Span_{\mathbb{Q}}\{v_k: k\in\mathbb{N}\}\subseteq\Span\{v_k: k\in\mathbb{N}\}$ is countable and dense, we conclude that $\Span\{v_k: k\in\mathbb{N}\}$ is separable and hence, it suffices to prove that
\begin{equation}
\label{tqi-348}
\Span\{v_k: k\in\mathbb{N}\}\cap B\,\,\mbox{ is dense in $B$.}
\end{equation}
Let $v\in B$. Then, there exists a sequence $\{v_{k_i}\}_i\subseteq B$ such that $T(v_{k_i})\to T(v)$ in $W$ as $i\to\infty$. Since $\{v_{k_i}\}_i$ is bounded in $V$ and $V$ is reflexive, by passing to a subsequence, we can assume that $\{v_{k_i}\}_i$ converges weakly in $V$ to some $\tilde{v}\in V$. Then Mazur's lemma yields a sequence of convex combinations that converge to $\tilde{v}$ in the norm on $V$:
\begin{equation}
\label{rhq-26}
\Span\{v_k: k\in\mathbb{N}\}\cap B\ni\sum_{j=i}^{L(i)}\alpha_{i,j}v_{k_j}\to \tilde{v}\,\,\,\mbox{ in $V$ as $i\to\infty$,}
\end{equation}
where $\alpha_{i,j}\geq 0$ and $\sum_{j=i}^{L(i)}\alpha_{i,j}=1$ with $L(i)\in\bbbn$. Appealing to the boundedness and linearity of $T$, we have
$$
T(\tilde{v})=\lim_{i\to\infty}T\Bigg(\sum_{j=i}^{L(i)}\alpha_{i,j}v_{k_j}\Bigg)
=\lim_{i\to\infty}\sum_{j=i}^{L(i)}\alpha_{i,j}T(v_{k_j})=T(v).
$$
Since $T$ is injective, we conclude that $\tilde{v}=v$ and \eqref{tqi-348} now follows from \eqref{rhq-26} because $v\in B$ was chosen arbitrarily.
\end{proof}

\begin{lemma}
\label{lipdensity}
Suppose that the space $(X,d,\mu)$ supports a $p$-Poincar\'e inequality for some $p\in [1,\infty)$. Then
the space $\operatorname{Lip}_b(X)$ of Lipschitz functions with bounded support
is a dense subset of $N^{1,p}(X)$.
\end{lemma}

For a proof see \cite[Corollary~5.15]{BjoBjo}. In fact, they proved density of compactly supported Lipschitz functions under the additional assumption that the space $X$ is complete. Without assuming completeness of  $X$, the same proof gives density of Lipschitz functions with bounded support. Completeness of $X$, since the space is equipped with a doubling measure, implies that bounded and closed sets are  compact  and hence Lipschitz functions with bounded support have compact support.

Lemma~\ref{lipdensity} follows also from Theorem~8.2.1 and the proof of Proposition~7.1.35 in \cite{HKST}.

We are now ready to present the

\begin{proof}[Proof of Theorem~\ref{apes}]
Suppose first that $p>1$. Note that $(X,d)$ is a separable metric space since $\mu$ is a doubling measure on $X$ and so, $L^p(X)$ is separable by \cite[Proposition~3.3.55]{HKST}. Clearly, the identity mapping $\iota\colon N^{1,p}(X)\to L^p(X)$ is a linear and bounded injective map. Now, since  the space $N^{1,p}(X)$ is reflexive by Proposition~\ref{F-UniCon}, Proposition~\ref{reflextosep} immediately implies that $N^{1,p}(X)$ is separable.

Suppose next that $p=1$ and fix $q\in(1,\infty)$. 
It follows from H\"older's inequality that $X$ supports a $q$-Poincar\'e inequality and by what we have already shown, $N^{1,q}(X)$ is separable so, there is a dense subset $\{\psi_i:\, i\in\bbbn\}$ of $N^{1,q}(X)$.

Fix $x_o\in X$ and for each $k\in\bbbn$ choose a Lipschitz function with bounded support $\eta_k\in\operatorname{Lip}_b(X)$ such that $\eta_k\equiv 1$ on $B(x_o,k)$. We will prove that $\mathcal{F}:=\{\eta_k\psi_i:\, k,i\in\bbbn\}$ is a dense subset of $N^{1,1}(X)$.

We first need to show that $\mathcal{F}\subseteq N^{1,1}(X)$. Fix $k,i\in\mathbb{N}$. It follows from Lemma~\ref{leibniz} that $\eta_k\psi_i\in N^{1,q}(X)$ and $h_{k,i}:=|
\eta_k|g_{\psi_i}+|\psi_i|\,{\rm lip}\,\eta_k$ is a $q$-weak (hence, also $1$-weak by Lemma~\ref{except}) upper gradient for $\eta_k\psi_i$, where $g_{\psi_i}\in L^q(X)$ is the minimal $q$-weak upper gradient for $\psi_i\in L^q(X)$. Since $\eta_k\in{\rm Lip}_b(X)$, it follows from \eqref{lillip} that ${\rm lip}\,\eta_k$ is a bounded function with bounded support. Therefore, by H\"older's inequality we can conclude that $\eta_k\psi_i, h_{k,i}\in L^{1}(X)$ and hence, $\eta_k\psi_i\in N^{1,1}(X)$, as wanted.

In light of Lemma~\ref{lipdensity} it suffices to prove that any Lipschitz function with bounded support  can be approximated in the $N^{1,1}$ norm by functions in $\mathcal{F}$.
Fix $u\in\operatorname{Lip}_b(X)$ and let $k_o\in\bbbn$ be such that $\supp u\subseteq B(x_o,k_o)$, so $\eta_{k_o}u=u$ pointwise in $X$. Since $u\in N^{1,q}(X)$, there is a sequence $\{\psi_{i_j}\}_j$ such that $\psi_{i_j}\to u$ in $N^{1,q}(X)$ as $j\to\infty$. Then it easily follows from Lemma~\ref{leibniz} that
$$
u-\eta_{k_o}\psi_{i_j}=\eta_{k_o}(u-\psi_{i_j})\to 0
\quad
\text{in $N^{1,1}(X)$ as $j\to\infty$.}
$$
This completes the proof of Theorem~\ref{apes}.
\end{proof}

\section{Pointwise estimates}
\label{ptwise}

The purpose of this section is to prove Theorem~\ref{main2}. In order to do so, it suffices to prove the following theorem.
\begin{theorem}
\label{ptwisethm}
Fix $p\in(1,\infty)$ and suppose that $X$ supports a $q$-Poincar\'e inequality for some $q\in[1,p)$. Then there exists a constant  $C\geq 1$ such that for all $u\in N^{1,p}(X)$,
\begin{equation}
\label{rqi-34}
C^{-1}g_u(x)\leq\limsup_{k\to\infty} |T_{k} u(x)|\leq Cg_u(x)\quad
\mbox{for $\mu$-a.e. $x\in X$,}
\end{equation}
where $g_u\in L^p(X)$ denotes the minimal $p$-weak upper gradient of $u$.
\end{theorem}

Indeed, Theorem~\ref{ptwisethm} and the following deep result due to Keith and Zhong \cite{KeithZhong} 
(see also \cite{Er}, \cite[Theorem~12.3.9]{HKST})
immediately yield Theorem~\ref{main2}.

\begin{lemma}[Keith and Zhong]
\label{KZ}
Let $(X,d,\mu)$ be a complete metric-measure space that supports a $p$-Poincar\'e inequality for some $p\in(1,\infty)$. Then there exists $q\in[1,p)$ such that $X$ supports a $q$-Poincar\'e inequality.
\end{lemma}

In the proof of Theorem~\ref{ptwisethm} we will make use of the \emph{Hardy-Littlewood maximal operator} of a function $g\in L^1_{{\rm loc}}(X)$ which is defined by
$$
\big(\mathscr{M}g\big)(x):=\sup_{r>0}\mvint_{B(x,r)}|g|\,d\mu\quad\mbox{for all $x\in X$.}
$$
We will use the boundedness of the maximal function in $L^p(X)$, \cite[Theorem~3.5.6]{HKST}:
\begin{lemma}
\label{T4}
If $\mu$ is a doubling measure on a metric space $X$ and $p\in (1,\infty]$, then
there is a constant $C$ depending on $p$ and the doubling constant of the measure only, such that
$\Vert \mathscr{M}g\Vert_{L^p(X)}\leq C\Vert g\Vert_{L^p(X)}$ for all $g\in L^p(X)$.
\end{lemma}

\begin{proof}[Proof of Theorem~\ref{ptwisethm}]
Assume that the $q$-Poincar\'e inequality holds with constants $c_{\PI}'$ and $\lambda'$.
Since all norms in $\bbbr^N$ are equivalent, it suffices to prove
that there exists a constant  $C\geq 1$ such that for all $u\in N^{1,p}(X)$,
\begin{equation}
\label{eq7}
C^{-1}g_u(x)\leq\limsup_{k\to\infty} |T_{k} u(x)|_p\leq Cg_u(x)\quad
\mbox{for $\mu$-a.e. $x\in X$.}
\end{equation}
Fix $u\in N^{1,p}(X)$. The second inequality in \eqref{eq7} follows from \eqref{tkpwest} and the Lebesgue differentiation theorem (Lemma~\ref{T5}) whenever $x$ is a Lebesgue point of $g_u^p$. Indeed, it is immediate from H\"older's inequality that $X$ supports a $p$-Poincar\'e inequality and so, Lemma~\ref{pwkpoin} implies the pair $(u,g_u)$ satisfies the $p$-Poincar\'e inequality \eqref{poincare}.

There remains to prove the first inequality in \eqref{eq7}. Our plan in this regard is to apply Lemma~\ref{pro:T_ku-converges-balls} with $h_k:=4\sup_{j\geq k}|T_ju|_p$ and  $h:=4\limsup_{k\to\infty} |T_{k} u|_p$ in order to conclude that $h$ is a $p$-weak upper gradient for $u$. To this end, first observe that clearly each $h_k$ and $h$ are nonnegative Borel functions. Moreover, Lemma~\ref{almostug} implies that if $d(x,y)\geq 2^{-k}$ for some $x,y\in X$ and $k\in\mathbb{N}$, and $\gamma$ is a rectifiable curve connecting $x$ and $y$, then
$$
|S_ku(x)-S_ku(y)|\leq \int_\gamma 4|T_ku|_p\,ds
\leq \int_\gamma h_k\,ds,
$$
where  $S_ku$ is as in \eqref{ukdef}. Hence, \eqref{bwr-249} in Lemma~\ref{pro:T_ku-converges-balls}  holds. 

Next, we claim that $\{h_k\}_{k\in\mathbb{N}}$ converges to $h$ in $L^p(X)$. 

Since $g_u$ is a $p$-weak upper gradient of $u$, it is also a $q$-weak upper gradient by Lemma~\ref{except}. Since $g_u$ is finite $\mu$-a.e., \eqref{tkpwest} in Lemma~\ref{TkBDD} (used here with $q$ in place of $p$) yields
\begin{equation}
\label{ejr-27}
|T_ku|_p\lesssim|T_ku|_q\lesssim
\Bigg(\,\,\mvint_{5\lambda' B[x]}g_u^q\,d\mu\Bigg)^{1/q}
\lesssim\big(\mathscr{M}g_u^q\big)^{1/q}\quad\mbox{pointwise on $X$,}
\end{equation}
where the implicit constant is independent of $k$. Note that in \eqref{ejr-27}, the first inequality is a consequence of the fact that all norms on $\bbbr^N$ are equivalent, and the last inequality follows from doubling condition and the definition of $\mathscr{M}$.  Therefore, we have that $h_k\lesssim\big(\mathscr{M}g_u^q\big)^{1/q}$ pointwise on $X$ for every $k\in\mathbb{N}$. On the other hand, since $g_u^q\in L^{p/q}(X)$ and $p/q>1$, the boundedness of $\mathscr{M}$ on $L^{p/q}(X)$ (Lemma~\ref{T4}) implies that $h_k\lesssim(\mathscr{M}g_u^q)^{1/q}\in L^p(X)$. Clearly, $\{h_k\}_{k\in\mathbb{N}}$ converges pointwise to $h$ and so, by Lebesgue's dominated convergence theorem, we have that $h_k\to h$ in $L^p(X)$. In particular,  $\{h_k\}_{k\in\mathbb{N}}$ converges weakly to $h$ in $L^p(X)$. Therefore, $\{h_k\}_{k\in\mathbb{N}}$ and $h$ satisfy the hypotheses of Lemma~\ref{pro:T_ku-converges-balls}, and it follows that $h$ is a $p$-weak upper gradient for ${u}$ which, in turn, implies that $h \ge g_u$ pointwise $\mu$-a.e.\@ in $X$ by the definition of a minimal $p$-weak upper gradient. This completes the proof of the first inequality in \eqref{eq7} and, in turn, the proof of Theorem~\ref{ptwisethm}.
\end{proof}


\begin{thebibliography}{00}
\frenchspacing
\setlength{\parskip}{0pt}
\setlength{\itemsep}{1 pt plus 0.5 pt minus 0.5 pt}

\bibitem{Ambrosio} 
Ambrosio, L., Colombo, M., Di Marino, S.:
Sobolev spaces in metric measure spaces: reflexivity and lower semicontinuity of slope, in \emph{Variational methods for evolving objects}, pp.\ 1--58, Adv.\ Stud.\ Pure Math., \textbf{67}, Math. Soc. Japan, 2015.

\bibitem{BjoBjo} Bj\"orn, A., Bj\"orn, J.:
\emph{Nonlinear Potential Theory on Metric Spaces},
EMS Tracts in Mathematics \textbf{17}, European Mathematical Society, Z\"urich, 2011.


\bibitem{Ch}
Cheeger, J.:
Differentiability of Lipschitz functions on metric measure spaces, {Geom.\ Funct.\ Anal.} \textbf{9} (1999), 428--517.

\bibitem{Clark}
Clarkson, J.A.:
Uniformly convex spaces,
{Trans. Amer. Math. Soc.} 
\textbf{40} (1936), no. 3, 396–414.

\bibitem{DurSha}
Durand-Cartagena, E., Shanmugalingam, N.: An elementary proof of Cheeger's theorem on reflexivity of Newton-Sobolev spaces of functions in metric measure spaces, {J. Anal.} \textbf{21} (2013), 73--83.

\bibitem{Er}
Eriksson-Bique, S.: Alternative proof of Keith-Zhong self-improvement and connectivity. 
{\em Ann.\ Acad.\ Sci.\ Fenn.\ Math.} \textbf{44} (2019), 407--425.

\bibitem{ES}
Eriksson-Bique, S., Soultanis, E.:
Curvewise characterizations of minimal upper gradients and the construction of a Sobolev differential. {Analysis and PDE} (to appear).

\bibitem{FL07} Fonseca, I., Leoni, G.: \textit{Modern methods in the calculus of variations: $L^p$ spaces,} Springer Monographs in Mathematics. Springer, New York, 2007.

\bibitem{FraHajKos}
  Franchi, B., Haj\l{}asz, P., Koskela, P.: Definitions of Sobolev classes on metric spaces, {Ann.\ Inst.\ Fourier} \textbf{49} (1999), 1903--1924.

\bibitem{Haj}
  Haj\l{}asz, P.: Sobolev spaces on metric-measure spaces, in \emph{Heat Kernels and Analysis on Manifolds, Graphs, and Metric Spaces (Paris, 2002)}, Contemp. Math. \textbf{338}, pp. 173--218, Amer. Math. Soc., Providence, RI, 2003.

\bibitem{Haj2}
Haj\l{}asz, P.:
Sobolev spaces on an arbitrary metric space, 
{Potential Anal.} \textbf{5} (1996), 403--415.


\bibitem{Hei07}
Heinonen, J.:
Nonsmooth calculus, 
{Bull.\ Amer.\ Math.\ Soc.\ (N.S.)} \textbf{44} (2007), 163--232.

\bibitem{HK}
Heinonen, J., Koskela, P.:
 Quasiconformal maps in metric spaces with controlled geometry, {Acta Math.} \textbf{181} (1998), 1--61.

\bibitem{HKST}
  Heinonen, J., Koskela, P., Shanmugalingam, N., Tyson, J.:
  \emph{Sobolev Spaces on Metric Measure Spaces: an Approach Based on Upper Gradients}, New Mathematical Monographs \textbf{27}, Cambridge University Press, 2015.

\bibitem{Keith}
 Keith, S.: A differentiable structure for metric measure spaces, {Adv.\ Math.} \textbf{183} (2004), 271--315. 

\bibitem{KeithZhong} Keith S. and Zhong X.: The Poincar\'e inequality is an open ended condition, {Ann. of Math.} \textbf{167} (2008), 575--599.

\bibitem{Sha}
Shanmugalingam, N.: Newtonian spaces: An extension of Sobolev spaces to metric measure spaces, {Rev. Mat. Iberoamericana} \textbf{16} (2000), 243--279.


\end{thebibliography}
\end{document}